\documentclass[final, 12pt,reqno]{amsart}

\usepackage{amssymb,amsmath,graphicx,amsfonts,euscript}
\usepackage{color}
\usepackage{showkeys}

\newtheorem{thm}{Theorem}[section]

\newtheorem{prop}[thm]{Proposition}

\newtheorem{lemma}[thm]{Lemma}

\def\om{\omega}

\def\pp{\partial}

\newcommand{\p}{\partial}

\newcommand{\beq}{\begin{equation}}
\newcommand{\eeq}{\end{equation}}
\newcommand{\ben}{\begin{eqnarray}}
\newcommand{\een}{\end{eqnarray}}
\newcommand{\beno}{\begin{eqnarray*}}
\newcommand{\eeno}{\end{eqnarray*}}

\numberwithin{equation}{section}

\begin{document}
\title[The 2D MHD equations]
{Global regularity for the 2D MHD equations with partial hyperresistivity}

\author[Dong, Li and Wu]{Bo-Qing Dong$^{1}$, Jingna Li$^{2}$ and  Jiahong Wu$^{3}$}

\address{$^1$ College of Mathematics and Statistics, Shenzhen University, Shenzhen 518060, China}
\email{bqdong@szu.edu.cn}

\address{$^3$ Department of Mathematics,
Jinan University, Guangzhou 510632, China
}

\email{jingna8005@hotmail.com}

\address{$^4$ Department of Mathematics, Oklahoma State University,
Stillwater, OK 74078}

\email{jiahong.wu@okstate.edu}

\vskip .2in
\begin{abstract}
This paper establishes the global existence and regularity for a system of
the two-dimensional (2D) magnetohydrodynamic (MHD) equations with only
directional hyperresistivity. More precisely, the equation of $b_1$ (the horizontal component of the magnetic field) involves only vertical hyperdiffusion
(given by $\Lambda_2^{2\beta} b_1$) while the equation of $b_2$ (the vertical component) has only horizontal hyperdiffusion (given by $\Lambda_1^{2\beta} b_2$), where $\Lambda_1$ and $\Lambda_2$ are directional Fourier multiplier operators
with the symbols being $|\xi_1|$ and $|\xi_2|$, respectively. We prove that, for $\beta>1$, this system always possesses a unique global-in-time classical solution
when the initial data is sufficiently smooth. The model concerned here is rooted
in the MHD equations with only magnetic diffusion, which play a significant role in the study of magnetic reconnection and magnetic turbulence.
In certain physical regimes and under suitable scaling,
the magnetic diffusion becomes partial (given by part of the Laplacian
operator). There have been considerable recent developments on the fundamental issue of whether classical solutions of these equations
remain smooth for all time. The papers of Cao-Wu-Yuan \cite{CaoWuYuan} and of Jiu-Zhao \cite{JiuZhao2}
obtained the global regularity when the magnetic diffusion is given by
the full fractional Laplacian $(-\Delta)^\beta$ with $\beta>1$. The main
result presented in this paper requires only directional fractional diffusion
and yet we prove the regularization in all directions.  The proof makes use of a key observation on the structure of the nonlinearity in the MHD equations and technical tools on Fourier multiplier operators such as the H\"{o}rmander-Mikhlin multiplier theorem. The result presented here appears to be the sharpest for the 2D MHD equations with partial magnetic diffusion.
\end{abstract}

\maketitle

\vskip .3in
\section{Introduction}

The magnetohydrodynamic (MHD) equations are the center piece of the magnetohydrodynamics. Since their initial derivation by the Nobel Laureate H. Alfv\'{e}n in 1924, the MHD equations have played pivotal roles in the study of many phenomena in geophysics, astrophysics, cosmology and engineering (see, e.g., \cite{Bis,Davi}). The standard incompressible MHD equations can be written as
\begin{equation}\label{MHD0}
 \left\{
\begin{array}{l}
u_t + u\cdot \nabla u  = -\nabla p + \nu \Delta u + b\cdot\nabla b, \\
b_t + u\cdot \nabla b = \eta\Delta b + b\cdot\nabla u,\\
\nabla\cdot u =0, \quad \nabla\cdot b =0,
 \end{array} \right.
 \end{equation}
where $u$ denotes the velocity field, $b$ the magnetic field,
$p$ the pressure, $\nu\ge 0$ the kinematic viscosity
and $\eta\ge 0$ the magnetic
diffusivity. (\ref{MHD0}) reflects the interaction between the velocity
field and the magnetic field. They consist of a coupled
system of the Navier-Stokes equations of fluid dynamics
and Maxwell's equations of electromagnetism.

\vskip .1in
The MHD equations are also of great interest in mathematics. Fundamental
issues such as the global existence and regularity of solutions to the MHD
equations have recently attracted considerable interest. Mathematically
the MHD equations are not merely a combination of two parallel
Navier-Stokes type equations but an interactive and integrated system. They
contain richer structures than the Navier-Stokes equations and exploring these
special structures can lead to interesting results that are not parallel to those
for the Navier-Stokes equations.

\vskip .1in
Attention here is focused on the 2D MHD equations. When there is no
kinematic dissipation or magnetic diffusion, namely (\ref{MHD0})
with $\nu=\eta=0$, the MHD equations become inviscid and the global
regularity problem appears to be out of reach at this moment. In contrast,
when both the dissipation and the magnetic diffusion are present,
namely (\ref{MHD0}) with $\nu>0$ and $\eta>0$, the MHD
equations are fully dissipative and the global regularity problem
in the 2D case can be solved following the approach for the 2D Navier-Stokes equations.  It is natural to explore the
intermediate equations that bridge the two extreme cases. The MHD equations
with partial or fractional dissipation exactly fill this gap. There have been
significant recent developments on the MHD
equations with partial or fractional dissipation. Important progress has
been made (see, e.g, \cite{Sc07, CaoReWu, CaoReWuZ, CaoWu, CaoWuYuan, CaiLei, Chem, ChenWang, Con22,DZ, FNZ, Fefferman1, Fefferman2, HeXuYu, HuX, HuLin, HuWang, JiuNiu, JNW, JiuZhao, JiuZhao2, LeiZ, LZ,LinZhang1, Ren, TrYu, WeiZ, Wu2,Wu3, Wu4, WuWu,WuWuXu,WuZhang,
Yam1, Yam2, Yam3, Yam4, Yam5, YuanBai, ZZ, TZhang}).

\vskip .1in
One special partial dissipation case is the 2D resistive MHD equations, namely
\begin{equation}\label{iMHD}
 \left\{
\begin{array}{l}
u_t + u\cdot \nabla u  = -\nabla p + b\cdot\nabla b, \\
b_t + u\cdot \nabla b = \eta\, \Delta b  + b\cdot\nabla u,\\
\nabla\cdot u =0, \quad \nabla\cdot b =0,
 \end{array} \right.
 \end{equation}
where $\eta> 0$ denotes the magnetic diffusivity (resistivity). (\ref{iMHD}) is applicable when
the fluid viscosity can be ignored while the role of
resistivity is important such as in magnetic reconnection
and magnetic turbulence. Magnetic reconnection refers to the
breaking and reconnecting of oppositely directed magnetic field
lines in a plasma and is at the heart of many spectacular events in
our solar system such as solar flares and northern
lights. The mathematical study of (\ref{iMHD}) may help
understand the Sweet-Parker
model arising in magnetic reconnection theory \cite{Pri}.
Although the global regularity problem on (\ref{iMHD}) is not completely solved at
this moment, recent efforts on this problem have significantly advanced our
understanding.

\vskip .1in
In certain physical regimes and under suitable scaling, the full
Laplacian dissipation is reduced to a partial dissipation. One notable example
is the Prandtl boundary layer equation in which only the vertical
dissipation is included in the horizontal component (see, e.g., \cite{Stew}). This paper focuses on a system of the 2D MHD equations that is closely related to (\ref{iMHD}),
\begin{equation}\label{MHD}
\left\{\aligned
&\partial_{t}u+(u \cdot \nabla) u +\nabla p= b\cdot \nabla b, \\
&\partial_{t} b_1 +(u \cdot \nabla) b_1 + \eta\Lambda_{2}^{2\beta}b_1=u\cdot\nabla b_1, \\
&\partial_{t} b_2 +(u \cdot \nabla) b_2 + \eta \Lambda_{1}^{2\beta}b_2=u\cdot\nabla b_2, \\
&\nabla\cdot u=0, \quad \nabla\cdot b=0, \\
&u(x, 0)=u_{0}(x),  \quad b(x,0)=b_{0}(x),
\endaligned\right.
\end{equation}
where $b_1$ and $b_2$ denote the components of $b$, $\eta>0$ and $\beta>0$ are real parameters.
The fractional partial derivative operators $\Lambda_{1}^{\gamma}$ and $\Lambda_{2}^{\gamma}$ with $\gamma>0$ are defined through the Fourier transform, namely
$$
\widehat{\Lambda_{1}^{\gamma}
f}(\xi_1, \xi_2)=|\xi_{1}|^{\gamma}\widehat{f}(\xi_1, \xi_2), \qquad
\widehat{\Lambda_{2}^{\gamma}
f}(\xi_1, \xi_2)=|\xi_2|^{\gamma}\widehat{f}(\xi_1, \xi_2).
$$
In addition, we also use $\Lambda^\sigma$ with $\sigma>0$ to denotes the 2D fractional Laplace operator,
$$
\widehat{\Lambda^{\sigma}
f}(\xi)=|\xi|^{\sigma}\widehat{f}(\xi), \qquad \xi =(
\xi_1, \xi_2).
$$
In comparison with (\ref{iMHD}), (\ref{MHD}) only has vertical fractional Laplacian
diffusion, no horizontal diffusion in the $b_1$ equation and no vertical
diffusion in the $b_2$ equation.

\vskip .1in
Our goal here is to show that, when $\beta>1$,  any sufficiently smooth initial
data $(u_0, b_0)$ leads to a unique global solution of (\ref{MHD}). More precisely, we establish the following theorem.

\begin{thm} \label{main}
Consider the 2D MHD equations in (\ref{MHD}) with $\eta>0$ and $\beta>1$. Assume
$(u_0, b_0) \in H^s(\mathbb{R}^2)$ with $s>2$, and $\nabla\cdot u_0 =0$ and $\nabla\cdot b_0 =0$. Then (\ref{MHD}) possesses a unique global solution $(u, b)$ satisfying,
for any $T>0$,
$$
(u,b) \in L^\infty(0, T; H^s(\mathbb{R}^2)), \quad b \in L^2(0, T; \dot{H}^{s+\beta}(\mathbb{R}^2)).
$$
\end{thm}

\vskip .1in
The proof of this result takes advantage of the special structure in the nonlinear terms of the magnetic field equation. Even though the system contains only directional fractional magnetic diffusion, we are still able to establish the
regularization in all directions. Theorem \ref{main} improves previous work of
Cao-Wu-Yuan \cite{CaoWuYuan} and of Jiu-Zhao \cite{JiuZhao2}. \cite{CaoWuYuan} and
\cite{JiuZhao2} obtained via different approaches the global regularity of a more regularized system
\begin{equation}\label{MHD2}
\left\{\aligned
&\partial_{t}u+(u \cdot \nabla) u +\nabla p= b\cdot \nabla b, \\
&\partial_{t} b +(u \cdot \nabla) b  + \eta\Lambda^{2\beta}b =u\cdot\nabla b, \\
&\nabla\cdot u=0, \quad \nabla\cdot b=0.
\endaligned\right.
\end{equation}
for the case when $\beta>1$. (\ref{MHD2}) involves full fractional
magnetic diffusion while (\ref{MHD}) involves only directional
fractional diffusion. The improvement is not a trivial one.

\vskip .1in
The proof of Theorem \ref{main} is not a simple generalization
of those in the previous papers \cite{CaoWuYuan} and
\cite{JiuZhao2}. Since (\ref{MHD}) contains only partial fractional diffusion,
some of the classical tools such as the maximal
regularity type estimates for the 2D heat equation can no longer be used here.
The equation of $b_1$ in (\ref{MHD}) involves only the vertical fractional
diffusion and in general we would not be able to obtain the smoothing of $b_1$
in the horizontal direction. However, the special
nonlinear structure actually allows us to prove that the derivatives of $b_1$ with respect to $x_1$ are globally bounded. The key observation is the
following identity, thanks to $\nabla\cdot u=0$ and $\nabla\cdot b=0$,
\begin{equation}\label{vv0}
b\cdot \nabla u_1 - u\cdot\nabla b_1 = \pp_1(b_1 u_1) + \pp_2(b_2 u_1) - \pp_1(b_1 u_1) - \pp_2 (u_2 b_1) =\pp_2(b_2 u_1 -u_2 b_1).
\end{equation}
To make use of this special structure, we write $b_1$ in the integral form,
\begin{equation}\label{vv1}
b_1(t) = g(t) \ast_2 b_{01} + \int_0^t g(t-\tau)\ast_2 (b\cdot\nabla u_1 - u\cdot\nabla b_1)\,d\tau,
\end{equation}
where $g$ denotes the 1D kernel functions associated with
the Fourier multiplier $e^{-t\,|\xi_2|^{2\beta}}$, namely
$$
g(x_2, t) = \int_{\mathbb{R}} e^{-t\,|\xi_2|^{2\beta}}\, e^{ix_2 \xi_2}\,d\xi_2.
$$
The notation for the convolution here is given by
\beno
g(t) \ast_2 b_{01} = \int_{\mathbb{R}} g(x_2- y_2, t) \,b_{01}(x_1, y_2)\,dy_2.
\eeno
(\ref{vv0}) and (\ref{vv1}) together allow us to obtain the control on $\pp_1 b_1$
and $\pp_1\pp_2 b_1$. Similarly, we can control $\pp_2 b_2$
and $\pp_1\pp_2 b_2$ even when the equation of $b_2$ involves no
vertical dissipation. This explains how we take advantage of the special
structure in the nonlinearity to control all the second-order derivatives
of $b$. To control even higher-order derivatives of $b$, say
$\Lambda^\sigma \Delta b$ with a fractional power $\sigma>0$, we first make use of the directional diffusion to obtain the directional regularization and then use the H\"{o}rmander-Mikhlin multiplier theorem to obtain the regularization in directions in which the directional diffusion is missing. More technical details can be found in the proof of Theorem \ref{main} in Section \ref{proofmain}.

\vskip .1in
The rest of this paper contains the proof of Theorem \ref{main}.
Section \ref{proofmain} is divided into three subsections which successively
provide more and more regular global bounds.

\vskip .3in
\section{Proof of Theorem \ref{main}}
\label{proofmain}
\setcounter{equation}{0}

This section proves Theorem \ref{main}. As we know, the core part of the
proof is the global {\it a priori} bounds. For the sake of clarity, we divide this
section into three subsections. The first subsection supplies the global $H^1$-bound, which relies on the equations of the vorticity $\omega$ and the current density $j=\nabla \times b$. The second subsection proves the global bounds for
$\|\nabla b\|_{L^\infty_t L^q}$ with any $1<q<\infty$, and for $\|\Delta b\|_{L^1_t L^q}$ and $\|\omega\|_{L^\infty_t L^q}$. The proof makes use of the
special structure of the nonlinear terms and a lemma assessing the behavior
of the 1D kernel function on Lebesgue spaces. The third subsection establishes
the global bounds for $\|\nabla j\|_{L^1_t L^\infty}$ and $\|\omega\|_{L^\infty_t L^\infty}$. To prose these global bounds, the key is to show that $\|\Lambda^\sigma
\Delta b\|_{L^1_t L^q}$ is globally bounded for any $0<\sigma< 2\beta-2$. We need to overcome the difficulty due to the lack of full fractional diffusion. The strategy
here is to first obtain the regularization along the direction of the
diffusion and then make use of the H\"{o}rmander-Mikhlin multiplier theorem to obtain the regularization in other directions.

\vskip .1in
\subsection{Global $H^1$ bound for $(u,b)$} This subsection proves that $(u, b)$
admits the following global $H^1$-bound.

\begin{prop} \label{h1bound}
Assume $(u_0, b_0)$ obeys the conditions stated in Theorem \ref{main}.
Let $(u, b)$ be the corresponding solution of (\ref{MHD}). Then $(u, b)$ satisfies
\ben
&& \hskip -.2in \|(u,b)\|_{L^2}^2
+ 2 \eta\, \int_0^t H(b)(\tau) \,d\tau
= \|(u_0,b_0)\|_{L^2}^2, \label{L2es}\\
&&  \hskip -.2in \|(\om, j)\|_{L^2}^2 + \eta \int_0^t \,H(\nabla b) (\tau)\, d\tau \notag\\
&&\qquad
\le C\,(1+\|(\om_0, j_0)\|_{L^2}) \exp\left(C\,(1+t)\,\|(u_0,b_0)\|_{L^2}^2\right),
 \label{h1es}\een
where $C=C(\beta)$ is a constant and
$$
H(b) =\|(\Lambda_2^\beta  b_1, \Lambda_1^\beta b_2)\|_{L^2}^2.
$$
\end{prop}

\begin{proof}
The global $L^2$ bound is obvious. Dotting (\ref{MHD}) with $(u, b)$,
integrating by parts and using
$\nabla\cdot u=0$ and $\nabla\cdot b=0$, we obtain (\ref{L2es}).
To obtain the global $H^1$ bound, we use the equations of the
vorticity $\om=\nabla\times u$ and the current density $j= \nabla\times b$,
\begin{equation}\label{omj}
\left\{\aligned
&\partial_{t} \om +(u \cdot \nabla) \om = b\cdot \nabla j,\\
&\partial_t j + u\cdot \nabla j + (\Lambda_1^{2\beta} \pp_1 b_2-\Lambda_2^{2\beta} \pp_2 b_1)=b\cdot\nabla\om + Q(u,b),
\endaligned\right.
\end{equation}
where
$$
Q(u,b) = 2 \pp_1 b_1(\pp_2 u_1+ \pp_1 u_2) - 2 \pp_1 u_1(\pp_2 b_1+ \pp_1 b_2).
$$
Integrating by parts and using $\nabla\cdot u=0$ and $\nabla\cdot b=0$, we have
\begin{equation}\label{h1}
\frac12 \frac{d}{dt} \|(\om, j)\|_{L^2}^2 + \eta \,H(\nabla b) = I,
\end{equation}
where
$$
H(\nabla b) = \| \Lambda_2^\beta \nabla b_1\|_{L^2}^2 + \| \Lambda_1^\beta\nabla b_2\|_{L^2}^2, \qquad I = \int Q(u,b)\, j\,dx.
$$
It suffices to estimate a typical term in $I$,
$$
I_1 = 2 \int \pp_1 b_1\, \pp_2 u_1 \, j\, dx
$$
By the boundedness of Zygmund-Calderon operators and a standard
Sobolev inequality,
$$
|I_1| \le C\, \|j\|^2_{L^4} \,\|\om\|_{L^2}
\le \frac{\beta \eta}{64} \|\nabla j\|_{L^2}^2
+ C\, \|j\|_{L^2}^2 \|\om\|_{L^2}^2.
$$
Due to the elementary inequality, for $\beta>1$,
$$
\xi_k^2 \le \frac{\beta-1}{\beta} + \frac1\beta \xi_k^{2\beta},\quad k=1,2,
$$
we have
$$
\|j\|_{L^2}^2 \le 2\,(\|\partial_1 b_2\|_{L^2}^2 + \|\partial_2 b_1\|_{L^2}^2) \le \frac{2(\beta-1)}{\beta} \|b\|_{L^2}^2 +
\frac2\beta H(b)
$$
and
$$
\|\nabla j\|_{L^2}^2 \le \frac{2(\beta-1)}{\beta} \|\nabla b\|_{L^2}^2 + \frac2\beta H(\nabla b).
$$
Inserting the bounds above in (\ref{h1}) yields
$$
\frac{d}{dt} \|(\om, j)\|_{L^2}^2 + \eta \,H(\nabla b) \le
\,C \|j\|^2_{L^2} (1+ \|\om\|_{L^2}^2),
$$
which yields the global $H^1$ bound
$$
\|(\om, j)\|_{L^2}^2 + \nu \int_0^t \,H(\nabla b) \, d\tau
\le C\,(1+\|(\om_0, j_0)\|^2_{L^2}) \exp{\int_0^t \|j\|_{L^2}^2\,d\tau}.
$$
This completes the proof of Proposition \ref{h1bound}.
\end{proof}

\vskip .1in
\subsection{Global bounds on $\|\nabla b\|_{L^\infty_t L^q}$, $\|\Delta b\|_{L^1_t L^q}$ and $\|\omega\|_{L^\infty_t L^q}$ with $1<q<\infty$}
This subsection uses the integral form of
the equation of $b$ to prove the following proposition.

\begin{prop} \label{good}
Assume $(u_0, b_0)$ obeys the conditions stated in Theorem \ref{main}.
Let $(u, b)$ be the corresponding solution of (\ref{MHD}) with $\beta>1$. Then $(u, b)$ obeys the
following global {\it a priori} bounds:
\begin{enumerate}
\item[(1)] For $1< q < \infty$ and $t>0$,
\begin{equation}\label{good1}
\|\nabla b(t)\|_{L^q(\mathbb{R}^2)} \le C(t, u_0, b_0), \quad \|j(t)\|_{L^q(\mathbb{R}^2)} \le C(t, u_0, b_0).
\end{equation}
A special consequence is the $L^\infty_t L^\infty(\mathbb{R}^2)$ bound for $b$,
$$
\|b\|_{L^\infty_t (L^\infty(\mathbb{R}^2))} \le C\,(\|b\|_{L^\infty_t L^2} + \|\nabla b\|_{L^\infty_t L^q}) = C(t, u_0, b_0).
$$
\item[(2)] For any $1<q<\infty$ and $t>0$,
\begin{equation}\label{good2}
\|\om\|_{L^\infty_t L^q} \le C(t, u_0, b_0), \quad \|\nabla j\|_{L^1_t L^q} \le C(t, u_0, b_0).
\end{equation}
\end{enumerate}
\end{prop}

\vskip .1in
To prove Proposition \ref{good}, we state a few properties for the kernel function
associated with the 1D fractional heat operator.
\begin{lemma} \label{jjj}
Let $\beta \ge 1$. Let $t>0$ and denote by $g=g(x_2, t)$ the 1D inverse Fourier transform of $e^{-|\xi_2|^{2\beta} t}$, namely
\begin{equation}\label{gdef}
g(x_2, t) = \int_{\mathbb{R}} e^{-t\,|\xi_2|^{2\beta}}\, e^{ix_2 \xi_2}\,d\xi_2.
\end{equation}
Then $g$ satisfies the following properties:
\begin{enumerate}
\item[(a)] For any $t>0$,
$$
g(x_2, t) = t^{-\frac1{2\beta}} \, g\left(\frac{x_2}{t^{\frac1{2\beta}}}, 1\right).
$$
\item[(b)] For any integer $m\ge 0$, any $1\le r\le \infty$ and any $t>0$,
\begin{equation}\label{gm}
\|\partial^m_{x_2} g(x_2, t)\|_{L^r(\mathbb{R})} \le C\, t^{-\frac{m}{2\beta}-\frac{1}{2\beta}(1-\frac1r)},
\end{equation}
which especially implies, for any $1\le p\le q\le \infty$ and  $f\in L^p(\mathbb{R})$,
$$
\|(\partial^m_{x_2} g(x_2, t))\ast_2 f\|_{L^q(\mathbb{R})} \le C\, t^{-\frac{m}{2\beta}-\frac{1}{2\beta}(\frac1p-\frac1q)}\,\|f\|_{L^p(\mathbb{R})}.
$$
\item[(c)] For any fractional $\sigma>0$, any $1 \le r \le\infty$ and any $t>0$,
$$
\|\Lambda^\sigma_{2} g(x_2, t)\|_{L^r(\mathbb{R})} \le C\, t^{-\frac{\sigma}{2\beta}-\frac{1}{2\beta}(1-\frac1r)},
$$
where $\Lambda^\sigma_2 g(x_2, t)$ is defined via the Fourier transform, for fixed $t>0$,
$$
\widehat{\Lambda^\sigma_2 g}(x_2, t) = |\xi_2|^\sigma e^{-|\xi_2|^{2\beta} t}.
$$
Especially,  for any $1\le p \le q \le \infty$ and  $f\in L^p(\mathbb{R})$,
$$
\|(\Lambda^\sigma_{x_2} g(x_2, t))\ast_2 f\|_{L^q(\mathbb{R})} \le C\, t^{-\frac{\sigma}{2\beta}-\frac{1}{2\beta}(\frac1p-\frac1q)}\,
\|f\|_{L^p(\mathbb{R})}.
$$
\end{enumerate}
\end{lemma}

\begin{proof}[Proof of Lemma \ref{jjj}] (a) follows directly from the
definition of $g$ in (\ref{gdef}). To prove (b), we first show that
the $L^1$-norm of $g(x_2, 1)$ is finite. In fact,
\begin{equation}\label{ff}
(1+ x_2^2)\,  g(x_2,1) = \int e^{ix_2 \xi_2}\, (1-\partial^2_{\xi_2})\,e^{-|\xi_2|^{2\beta}}\,d\xi_2.
\end{equation}
It is clear that, for $\beta \ge1$, the right-hand side of (\ref{ff}) is finite. Thus,
$$
|g(x_2, 1)| \le C\, (1+ x_2^2)^{-1} \quad\mbox{and} \quad \|g(\cdot,1)\|_{L^1(\mathbb{R})} \le C.
$$
According to (a), for any $t>0$,
$$
\|g(\cdot, t)\|_{L^1(\mathbb{R})} = \|g(\cdot,1)\|_{L^1(\mathbb{R})} \le C.
$$
For any $t>0$,
$$
\|g(\cdot, t)\|_{L^\infty(\mathbb{R})} \le \|e^{-t\,|\xi_2|^{2\beta}}\|_{L^1(\mathbb{R})} = C\, t^{-\frac1{2\beta}},
$$
where $C$ is a constant independent of $t$.
Therefore, for any $1\le r\le \infty$, by a simple interpolation inequality,
$$
\|g(\cdot, t)\|_{L^\infty(\mathbb{R})} \le \|g(\cdot, t)\|_{L^\infty(\mathbb{R})}^{1-\frac1r}\, \|g(\cdot, t)\|^{\frac1{r}}_{L^1(\mathbb{R})} \le C\,  t^{-\frac{1}{2\beta}(1-\frac1r)}.
$$
This proves (\ref{gm}) with $m=0$. The general case $m>0$ can be  shown
by repeating the process above with $\partial^m_{x_2} g(x_2, t)$ whose corresponding Fourier transform is  $\xi_2^m e^{-|\xi_2|^{2\beta} t}$. The proof
of the results in (c) for the fractional derivative $\sigma$ is similar.
We omit further details.
\end{proof}


\vskip .1in
\begin{proof}[Proof of Proposition \ref{good}]
To prove (\ref{good1}), we start with the integral representations of $b_1$
and $b_2$,
\ben
&& b_1 = g(t) \ast_2 b_{01} + \int_0^t g(t-\tau)\ast_2 (b\cdot\nabla u_1 - u\cdot\nabla b_1)\,d\tau, \label{b1in}\\
&& b_2 = h(t) \ast_1 b_{02} + \int_0^t h(t-\tau)\ast_1(b\cdot\nabla u_2 - u\cdot\nabla b_2)\,d\tau, \label{b2in}
\een
where $g$ and $h$ denote the 1D kernel functions associated with
the Fourier multiplier $e^{-t |\xi_2|^{2\beta}}$ and $e^{-t |\xi_1|^{2\beta}}$, namely
$$
g(x_2, t) = \int_{\mathbb{R}} e^{-t\,|\xi_2|^{2\beta}}\, e^{ix_2 \xi_2}\,d\xi_2,\qquad h(x_1, t) = \int_{\mathbb{R}} e^{-t\,|\xi_1|^{2\beta}}\, e^{ix_1 \xi_1}\,d\xi_1
$$
and the convolution notations are defined as
\beno
g(t) \ast_2 b_{01} = \int_{\mathbb{R}} g(x_2- y_2, t) \,b_{01}(x_1, y_2)\,dy_2, \\
h(t) \ast_1 b_{02} = \int_{\mathbb{R}} h(x_1- y_1, t) \,b_{01}(y_1, x_2)\,dy_1.
\eeno
To prove (\ref{good1}), it suffices to prove the bounds for
for $\pp_2 b_1$  and $\pp_1 b_2$. Since, if
$$
\|\pp_2 b_1\|_{L^q} \le C(t, u_0, b_0), \qquad
\|\pp_1 b_2\|_{L^q} \le C(t, u_0, b_0),
$$
then
$$
\|j\|_{L^q} \le C(t, u_0, b_0), \qquad
\|\nabla b\|_{L^q} \le C\, \|j\|_{L^q} \le C(t, u_0, b_0).
$$
To show the bound for
$\pp_2 b_1$, we write
\begin{equation}\label{sss}
b\cdot \nabla u_1 - u\cdot\nabla b_1 = \pp_1(b_1 u_1) + \pp_2(b_2 u_1) - \pp_1(b_1 u_1) - \pp_2 (u_2 b_1) =\pp_2(b_2 u_1 -u_2 b_1)
\end{equation}
and thus
$$
\pp_2 b_1 = \pp_2(g(t) \ast_2 b_{01}) + \int_0^t \pp_2 \p_2 g(t-\tau)\ast_2 (b_2 u_1 -u_2 b_1)(\tau)\,d\tau.
$$
First we take $L^q_{x_1}$ each side to obtain
$$
\|\pp_2 b_1\|_{L^q_{x_1}} \le |g(t)| \ast_2 \|\pp_2 b_{01}\|_{L^q_{x_1}} + \int_0^t |\pp_2 \p_2 g(t-\tau)|
\ast_2 \|(b_2 u_1 -u_2 b_1)(\tau)\|_{L^q_{x_1}}\, d\tau
$$
We then take $L^q_{x_2}$ each side and apply Young's inequality
for convolution to obtain
\begin{equation} \label{pp1}
\|\pp_2 b_1\|_{L^q} \le \|g(t)\|_{L^1} \,\|\pp_2 b_{01}\|_{L^q}
+ \int_0^t \|\pp_2 \p_2 g(t-\tau)\|_{L^1_{x_2}} \|b_2 u_1 -u_2 b_1
\|_{L^q}(\tau) \,d\tau.
\end{equation}
By Lemma \ref{jjj},
\begin{equation} \label{tb}
\|\pp_2 \p_2 g(t-\tau)\|_{L^1_{x_2}}  \le C\, (t-\tau)^{-\frac1\beta},
\end{equation}
where $C$ is a constant depending on $\beta$ only. By H\"{o}lder's inequality and Sobolev's inequality,
$$
\|b_2 u_1 -u_2 b_1
\|_{L^q} \le \|u\|_{L^{2q}} \,\|b\|_{L^{2q}}
\le C\,(\|u\|_{L^2} + \|\om\|_{L^2})\, (\|b\|_{L^2} + \|j\|_{L^2}).
$$
Inserting these estimates in (\ref{pp1}) yields
$$
\|\pp_2 b_1(t)\|_{L^q} \le C\,\|\pp_2 b_{01}\|_{L^q} + C\, t^{1-\frac1\beta} \, (\|u\|_{L^\infty_t L^2} +\|\om\|_{L^\infty_t L^2})\,(\|b\|_{L^\infty_t L^2} + \|j\|_{L^\infty_tL^2}).
$$
Similarly, for any $t>0$,
$$
\|\pp_1 b_2(t)\|_{L^q} \le C\,\|\pp_1 b_{02}\|_{L^q} + C\, t^{1-\frac1\beta} \,
(\|u\|_{L^\infty_t L^2} +\|\om\|_{L^\infty_t L^2})\,(\|b\|_{L^\infty_t L^2} + \|j\|_{L^\infty_tL^2}).
$$
Consequently
\beno
\|\nabla b\|_{L^\infty_tL^q} &\le& C\, \|j\|_{L^\infty_t L^q} \\
&\le& C\,\|\nabla b_0\|_{L^p} + C\, t^{1-\frac1\beta} \, (\|u\|_{L^\infty_t L^2} +\|\om\|_{L^\infty_t L^2})\,(\|b\|_{L^\infty_t L^2} + \|j\|_{L^\infty_tL^2}).
\eeno
An elementary Sobolev inequality then implies that, for any
$q>2$,
$$
\|b\|_{L^\infty_t L^\infty} \le C (\|b\|_{L^\infty_t L^2} + \|\nabla b\|_{L^\infty_t L^q})= C(t, u_0, b_0).
$$

\vskip .1in
Next we prove (\ref{good2}). To do so, we combine the estimates of
$\|\om\|_{L^\infty_t L^q}$ with $\|\nabla j\|_{L^1_t L^q}$. It follows from the vorticity equation (see (\ref{omj})) that
\begin{equation} \label{ooo}
\|\om(t)\|_{L^q} \le \|\om_0\|_{L^q} + \|b\|_{L^\infty_{x,t}} \int_0^t \|\nabla j(\tau)\|_{L^q}\,d\tau.
\end{equation}
We then bound $\|\pp_2 \pp_2 b_1\|_{L^q}$ and $\|\pp_1 \pp_2 b_2\|_{L^q}$ in terms of $\|\om\|_{L^q}$. Applying $\pp_2 \pp_2$ to (\ref{b1in}) yields
$$
\pp_2 \pp_2 b_1 = \pp_2 \pp_2 (g(t) \ast_2 b_{01}) + \int_0^t \pp_2 \pp_2 g(t-\tau)\ast_2 (b\cdot\nabla u_1 - u\cdot\nabla b_1)\,d\tau.
$$
As in the proof of (\ref{pp1}), we have
\beno
\|\pp_2 \pp_2 b_1\|_{L^q} &\le&
\|\pp_2\pp_2(g(t)\ast_2 b_{01})\|_{L^q}\\
&& + \int_0^t \|\pp_2\pp_2 g(t-\tau)\|_{L^1_{x_2}} \|b\cdot\nabla u_1 -u\cdot\nabla b_1)\|_{L^q} \,d\tau.
\eeno
Furthermore, by Sobolev's inequality,
\beno
\|b\cdot\nabla u_1 -u\cdot\nabla b_1)\|_{L^q} &\le& \|b\|_{L^\infty}
\|\om\|_{L^q} + \|u\|_{L^{2q}} \|\nabla b_1\|_{L^{2q}}\\
&\le& C\, \|\om\|_{L^q} + C\,(\|u\|_{L^2} + \|\omega\|_{L^2})(\|j\|_{L^2} + \|\nabla j\|_{L^2}).
\eeno
Taking $L^1$ in time on $[0, t]$, applying Young's inequality for convolution
and invoking (\ref{tb}), we have
$$
\|\pp_2 \pp_2 b_1
\|_{L^1_t L^q}
\le \|\pp_2\pp_2(g(t)\ast_2 b_{01})\|_{L^1_t L^q}
+ C\, \|\om\|_{L^1_t L^q} + C,
$$
where $C=C(t, u_0, b_0)$ is bound for the norms of $(u, b)$
obtained previously.  Clearly, by Young's inequality for convolution,
$$
\|\pp_2\pp_2(g(t)\ast_2 b_{01})\|_{L^1_t L^q}
\le \|\pp_2\pp_2\,g_2\|_{L^1_tL^1_{x_2}}\, \|b_{01}\|_{L^q}
\le C t^{1-\frac1\beta} \, \|b_{01}\|_{L^q}.
$$
Therefore,
$$
\|\pp_2 \pp_2 b_1
\|_{L^1_t L^q}  \le C t^{1-\frac1\beta} \, \|b_{01}\|_{L^q}
+ C\, t^{1-\frac1\beta} (\|\om\|_{L^1_t L^q} + 1).
$$
Similarly,
$$
\|\pp_1 \pp_1 b_2
\|_{L^1_t L^q}  \le C t^{1-\frac1\beta} \, \|b_{02}\|_{L^q}
+ C\, t^{1-\frac1\beta} (\|\om\|_{L^1_t L^q} + 1).
$$
To estimate $\|\pp_1 \pp_2 b_1
\|_{L^1_t L^q}$, we resort to the special structure of the nonlinear term, namely (\ref{sss}), which allows us to write
$\pp_1 \pp_2 b_1$ as
\ben
\pp_1 \pp_2 b_1(t) &=& \pp_1\pp_2(g(t)\ast_2 b_{01})
+ \int_0^t \pp_2 g(t-\tau)\ast_2 \pp_1 (b\cdot\nabla u_1 -u\cdot\nabla b_1) \,d\tau  \notag\\
&=& \pp_1\pp_2(g(t)\ast_2 b_{01})
+ \int_0^t  \pp_2 g(t-\tau)\ast_2 \pp_1\pp_2(b_2 u_1- u_2 b_1)\,d\tau   \notag\\
&=& \pp_1\pp_2(g(t)\ast_2 b_{01})
+ \int_0^t  \pp_2\pp_2 g(t-\tau)\ast_2 \pp_1(b_2 u_1- u_2 b_1)\,d\tau. \label{goo}
\een
Then, as in the estimate of $\pp_2\pp_2 b_1$, we have
\beno
\|\pp_1 \pp_2 b_1\|_{L^1_t L^q}  &\le& \|\pp_2 g(t)\|_{L^1_{x,t}}
\|\pp_1 b_{01}\|_{L^q} + C \, t^{1-\frac1\beta} (\|\om\|_{L^1_t L^q} + 1)\\
&\le& C t^{1-\frac1{2\beta}} \|\pp_1 b_{01}\|_{L^q} + C \, t^{1-\frac1\beta} (\|\om\|_{L^1_t L^q} + 1).
\eeno
Similarly,
$$
\|\pp_1 \pp_1 b_2\|_{L^1_t L^q} \le C t^{1-\frac1{2\beta}} \|\pp_2 b_{02}\|_{L^q} + C \, t^{1-\frac1\beta} (\|\om\|_{L^1_t L^q} + 1).
$$
Therefore, by $\nabla\cdot b=0$,
\ben
\|\nabla j\|_{L^1_tL^q}
&=& \|(\Delta b_2, -\Delta b_1)\|_{L^1_tL^q}  \notag\\
&\le& \|\pp_2\pp_2 b_1\|_{L^1_tL^q}
+ \|\pp_1\pp_2 b_2\|_{L^1_tL^q} + \|\pp_1 \pp_1 b_2
\|_{L^1_t L^q} + \|\pp_1 \pp_2 b_1\|_{L^1_t L^q}   \notag\\
&\le& C t^{1-\frac1{2\beta}} \|\nabla b_{0}\|_{L^q} + C \, t^{1-\frac1\beta} (\|\om\|_{L^1_t L^q} + 1). \label{jb}
\een
Then (\ref{ooo}) and (\ref{jb}), together with Gronwall's
inequality implies (\ref{good2}).
\end{proof}

\vskip .1in
\subsection{Global bounds for $\|\nabla j\|_{L^1_tL^\infty}$ and $\|\omega\|_{L^\infty_{x,t}}$ and proof of Theorem \ref{main}} This subsection proves that $\omega$ admits a
global bound in $L^\infty_{x,t}$. This
crucial global bound then ensures a global bound for $\|(u, b)\|_{H^s}$ for any $s>0$.

\vskip .1in
\begin{prop} \label{global3}
Assume $(u_0, b_0)$ satisfies the conditions stated in
Theorem \ref{main}. Let $(u, b)$ be the corresponding solution
of (\ref{MHD}) with $\beta>1$. Then, $(u, b)$
admits the following global bounds, for any $0<t<\infty$,
\begin{equation}\label{good3}
\|\nabla j\|_{L^1_t L^\infty_x} \le C(t, u_0, b_0), \qquad
\|\om\|_{L^\infty_{x,t}} \le C(t, u_0, b_0)
\end{equation}
and
\begin{equation}\label{gggg}
\|(u, b)\|_{H^s}\le C(t, u_0, b_0).
\end{equation}
\end{prop}

To prove Proposition \ref{global3}, we need the following H\"{o}rmander-Mikhlin multiplier theorem (see, e.g. \cite[p.96]{Stein}).
\begin{lemma} \label{mul}
Let $m$ be a bounded function on $\mathbb{R}^d$ which is smooth except possibly at the origin, and such that
$$
|\nabla^{k} m(\xi)| \le C\,|\xi|^{-k}, \qquad 0\le k\le \frac{d}{2}+1.
$$
Then m is an $L^p$ multiplier for all $1 < p < \infty$, or the operator $T_m$ defined by $$\widehat{T_m f} = m \widehat{f}, \qquad f\in L^2\cap L^p,
$$
is bounded from $L^2\cap L^p$ to $L^2\cap L^p$.
\end{lemma}

\begin{proof}[Proof of Proposition \ref{global3}]
Due to the embedding inequality, for any $q>\frac2\sigma$,
\beno
\|\nabla j\|_{L^\infty} \le C\,(\|\nabla j\|_{L^2} + \|\Lambda^\sigma \nabla j\|_{L^q}),
\eeno
it suffices to show that, for some $\sigma>0$ and for all
$2\le q<\infty$,
\ben \label{good30}
\|\Lambda^\sigma \nabla j\|_{L^1_t L^q} <\infty.
\een
We first show, for $0<\sigma < 2\beta -2$,
\ben \label{good31}
\|\Lambda_2^\sigma \pp_2\pp_2 b_1\|_{L^1_t L^q} \le C(t, u_0, b_0) <\infty, \quad
\|\Lambda_1^\sigma \pp_1\pp_1 b_2\|_{L^1_t L^q} \le C(t, u_0, b_0)<\infty.
\een
Applying $\Lambda_2^\sigma \pp_2\pp_2$ to the integral representation of $b_1$ in
(\ref{b1in}), taking the norm in $L^1_t L^q$ and using Young's inequality for convolution, we obtain
\beno
\|\Lambda_2^\sigma \pp_2\pp_2 b_1\|_{L^1_t L^q}
&\le& \|\Lambda_2^\sigma \pp_2\pp_2 g\|_{L^1_t L^1_x} \|b_{01}\|_{L^q}\\
&&   + \,\|\Lambda_2^\sigma \pp_2\pp_2 g\|_{L^1_t L^1_x}\,
\|b\cdot\nabla u_1 -u\cdot\nabla b_1\|_{L^1_t L^q_x}.
\eeno
According to Lemma \ref{jjj},
$$
\|\Lambda_2^\sigma \pp_2\pp_2 g\|_{L^1_x} \le C\, t^{-\frac{2+\sigma}{2\beta}}.
$$
Therefore, for $0<\sigma < 2\beta -2$,
$$
\|\Lambda_2^\sigma \pp_2\pp_2 g\|_{L^1_t L^1_x} = C\, t^{\frac{2\beta-(2+\sigma)}{2\beta}}.
$$
By H\"{o}lder's inequality and Sobolev's inequality,
$$
\|b\cdot\nabla u_1 -u\cdot\nabla b_1\|_{L^1_t L^q_x} \le C(t, u_0, b_0).
$$
Therefore,
$$
\|\Lambda_2^\sigma \pp_2\pp_2 b_1\|_{L^1_t L^q} \le C(t, u_0, b_0).
$$
Similarly,
$$
\|\Lambda_1^\sigma \pp_1\pp_1 b_2\|_{L^1_t L^q} \le C(t, u_0, b_0).
$$
Therefore, (\ref{good31}) holds.
Making use of the structure of the nonlinearity, we can also show that
\ben \label{good32}
\|\Lambda_2^\sigma \pp_1\pp_2 b_1\|_{L^1_t L^q} \le C(t, u_0, b_0), \qquad
\|\Lambda_1^\sigma \pp_1\pp_2 b_2\|_{L^1_t L^q} \le C(t, u_0, b_0).
\een
In fact, applying $\Lambda_2^\sigma \pp_1\pp_2$ to (\ref{b1in}), writing
$\pp_1\pp_2 b_1$ as in (\ref{goo}) and taking the norm in $L^1_t L^q$, we obtain
\beno
\|\Lambda_2^\sigma \pp_1\pp_2 b_1\|_{L^1_t L^q} &\le& \|\Lambda^\sigma_2\pp_2 g\|_{L^1_t L^1_x} \|\pp_1b_{01}\|_{L^q}\\
&& +\, \|\Lambda^\sigma_2\pp_2\pp_2 g\|_{L^1_t L^1_x} \|\pp_1(b_2 u_1
-u_2 b_1)\|_{L^1_t L^q_x} \le C <\infty.
\eeno
Therefore, (\ref{good32}) holds. Next we show that
\ben \label{good33}
\|\Lambda_1^\sigma \pp_2\pp_2 b_1\|_{L^1_t L^q} <\infty, \qquad
\|\Lambda_2^\sigma \pp_1\pp_1 b_2\|_{L^1_t L^q} <\infty.
\een
It appears that we can not prove (\ref{good33}) in the same way as
(\ref{good31}) and (\ref{good32}). The main reason is that, when we apply
the operator $\Lambda_1^\sigma \pp_2\pp_2$ to the integral representation
of $b_1$ in (\ref{b1in}), the part $\Lambda_1^\sigma$ has to be applied to $b\cdot\nabla u_1 -u\cdot\nabla b_1$, but unfortunately we have no control on
$\Lambda_1^\sigma\nabla u_1$. Instead we prove (\ref{good33}) using the
the H\"{o}rmander-Mikhlin multiplier theorem stated in Lemma \ref{mul}. More
precisely, due to the simple inequality
$$
|\xi_1|^\sigma |\xi_2|^2 \le \frac{2}{2+\sigma} (\xi_2^2)^{1+\frac\sigma2} +
\frac{\sigma}{2+\sigma} (\xi_1^2)^{1+\frac\sigma2},
$$
Plancherel's theorem and the global bounds in (\ref{good31}) and (\ref{good32}) imply that
$$
\|\Lambda_1^\sigma \pp_2\pp_2 b_1\|_{L^1_t L^2}
\le C\,\left(\|\Lambda_2^\sigma \pp_2\pp_2 b_1\|_{L^1_t L^2} + \|\Lambda_1^\sigma \pp_1\pp_2 b_2\|_{L^1_t L^2}\right) \le C(t, u_0, b_0).
$$
Define the Fourier multiplier operator $T_m$ by
$$
\widehat{T_m f}(\xi) = m(\xi) \, \widehat{f}(\xi), \qquad m(\xi) = \frac{|\xi_1|^\sigma |\xi_2|^2}{(\xi_2^2)^{1+\frac\sigma2} + (\xi_1^2)^{1+\frac\sigma2}}.
$$
It is easy to check that $m$ obeys the conditions of Lemma \ref{mul}. It then
follows from Lemma \ref{mul} that
\beno
\|\Lambda_1^\sigma \pp_2\pp_2 b_1\|_{L^q} &=& \|T_m \Lambda_2^\sigma \pp_2\pp_2 b_1
+ T_m \Lambda_1^\sigma \pp_1\pp_2 b_2\|_{L^q}\\
 &\le& C\, \left(\|\Lambda_2^\sigma \pp_2\pp_2 b_1\|_{L^q} + \|\Lambda_1^\sigma \pp_1\pp_2 b_2\|_{L^q}\right).
\eeno
Therefore, the global bounds in (\ref{good31}) and (\ref{good32}) implies
$$
\|\Lambda_1^\sigma \pp_2\pp_2 b_1\|_{L^1_t L^q}
\le C\,\left(\|\Lambda_2^\sigma \pp_2\pp_2 b_1\|_{L^1_t L^q} + \|\Lambda_1^\sigma \pp_1\pp_2 b_2\|_{L^1_t L^q}\right) \le C(t, u_0, b_0).
$$
Similarly,
$$
\|\Lambda_2^\sigma \pp_1\pp_1 b_2\|_{L^1_t L^q} \le C(t, u_0, b_0).
$$
This proves (\ref{good33}). We can also prove in a similar fashion that
\ben \label{good34}
\|\Lambda_1^\sigma \pp_1\pp_2 b_1\|_{L^1_t L^q} <\infty, \qquad
\|\Lambda_2^\sigma \pp_1\pp_2 b_2\|_{L^1_t L^q} <\infty.
\een
It is clear that (\ref{good31}), (\ref{good32}), (\ref{good33})
and (\ref{good34}) imply (\ref{good30}). We thus have obtained
$$
\|\nabla j\|_{L^1_t L^\infty} \le C (t, u_0, b_0).
$$
Furthermore, the vorticity equation implies
$$
\|\omega(t)\|_{L^\infty} \le \|\omega_0\|_{L^\infty} + \int_0^t \|b(\tau)\|_{L^\infty}\|\nabla j(\tau)\|_{L^\infty}\,d\tau \le C(t, u_0, b_0).
$$
This completes the proof of (\ref{good3}). Once we have the global bounds
$$
\|\omega\|_{L^1_t L^\infty_x} \le C (t, u_0, b_0), \quad \|j\|_{L^1_t L^\infty_x} \le C (t, u_0, b_0),
$$
the global bound in (\ref{gggg}) then follows from a standard procedure
(see, e.g., \cite{BKM}).
This completes the proof of Proposition \ref{global3}.
\end{proof}

\vskip .1in
We finally provide the proof of Theorem \ref{main}.

\begin{proof}[Proof of  Theorem \ref{main}] Once the global {\it a priori} bounds is at our disposal, the proof can be achieved via a standard procedure. First we seek the solution
of a regularized system. We begin by introducing a few notation. For $\varepsilon>0$, we denote by $\phi_\varepsilon$ the standard mollifier, namely
$$
\phi_\varepsilon(x)=\varepsilon^{-2}\phi(\varepsilon^{-1}|x|)
$$
with
$$
\phi\in C_0^\infty(\mathbb{R}^2), \quad \phi(x) = \phi(|x|), \quad \mbox{supp}\phi\subset \{x| |x|<1\}, \quad \int_{\mathbb{R}^2} \phi(x)\, dx =1.
$$
For any locally integrable function $v$, define the mollification $\mathcal{J}_\varepsilon v$ by
$$
\mathcal{J}_\varepsilon v = \phi_\varepsilon \ast v.
$$
Let $\mathbb{P}$ denote the Leray projection operator (onto divergence-free vector fields). We seek a solution $(u^\varepsilon, b^\varepsilon)$ of the system
$$
\left\{\begin{array}{l}
\partial_t u^\varepsilon  + \mathbb{P} \mathcal{J}_\varepsilon((\mathcal{J}_\varepsilon u^\varepsilon) \cdot\nabla (\mathcal{J}_\varepsilon u^\varepsilon))
          =    \mathbb{P}\mathcal{J}_\varepsilon((\mathcal{J}_\varepsilon b^\varepsilon) \cdot\nabla (\mathcal{J}_\varepsilon b^\varepsilon))  ,\vspace{2mm}\\
\partial_t b_1^\varepsilon  + \mathcal{J}_\varepsilon((\mathcal{J}_\varepsilon u^\varepsilon) \cdot\nabla (\mathcal{J}_\varepsilon b_1^\varepsilon))
            + \eta \mathcal{J}^2_\varepsilon\Lambda_2^{2\beta} b^\varepsilon _1
            =   \mathcal{J}_\varepsilon((\mathcal{J}_\varepsilon b^\varepsilon) \cdot\nabla (\mathcal{J}_\varepsilon u_1^\varepsilon)) ,\vspace{2mm}\\
\partial_t b^\varepsilon _2 + \mathcal{J}_\varepsilon((\mathcal{J}_\varepsilon u^\varepsilon) \cdot\nabla (\mathcal{J}_\varepsilon b^\varepsilon _2))
           + \eta \mathcal{J}^2_\varepsilon \Lambda_{1}^{2\beta} b_2^\varepsilon     =   \mathcal{J}_\varepsilon((\mathcal{J}_\varepsilon b^\varepsilon) \cdot\nabla (\mathcal{J}_\varepsilon u_2^\varepsilon)),\vspace{2mm}\\
\nabla \cdot u^\varepsilon  =\nabla \cdot b^\varepsilon  = 0,\\
(u^\varepsilon,w^\varepsilon)(x,0)=(u_0*\phi_\varepsilon,\ \
b_0*\phi_\varepsilon)=(u_0^\varepsilon,b_0^\varepsilon).
\end{array}\right.
$$
Following the lines as those in the proofs of Propositions
\ref{h1bound}, \ref{good} and   \ref{global3}, we can establish the
global bound, for any $t\in (0, \infty)$,
$$
\|u^\varepsilon(t)\|_{H^s}^2+\|b^\varepsilon(t)\|_{H^s}^2 \leq C(t, u_0, b_0).
$$
A standard compactness argument allows us to obtain the global
existence of the classical solution $(u,b)$ to (\ref{MHD}). The
uniqueness can also be easily established. We omit further details. This completes  the proof of Theorem \ref{main}.
\end{proof}

\vskip .4in \noindent \textbf{Acknowledgments}\,\, B. Dong was
supported by the NNSFC grants No.11271019 and 11571240. J. Li was
partially supported by the NNSFC grant No.11201181. J. Wu was partially
supported by the NSF grant DMS 1614246, by the AT\&T
Foundation at Oklahoma State University, and by NNSFC grant No.11471103
(a grant awarded to B. Yuan).


\vskip .3in
\begin{thebibliography}{999}

\bibitem{Sc07} R. Agapito and M. Schonbek, Non-uniform decay of MHD equations with and without magnetic diffusion, {\it Comm. Partial Differential Equations \bf32} (2007), 1791--1812.

\bibitem{Am}H. Amann,  Maximal regularity for nonautonoumous evolution
equations, {\it Adv. Nonlinear Stud. \bf 4}  (2004),  417--430.

\bibitem{BKM}J. Beale,  T. Kato and  A. Majda,
Remarks on the breakdown of smooth solutions for the 3-D Euler
equations, {\it Commun. Math. Phys.  \bf 94} (1984),  61--66.


\bibitem{Bis} D. Biskamp, {\it Nonlinear Magnetohydrodynamics},
Cambridge University Press, Cambridge, 1993.

\bibitem{CaoReWu} C. Cao, D. Regmi and J. Wu, The 2D MHD equations with horizontal dissipation and horizontal
magnetic diffusion, {\it J. Differential Equations \bf 254} (2013), 2661--2681.

\bibitem{CaoReWuZ} C. Cao, D. Regmi, J. Wu and X. Zheng, Global regularity for the 2D magnetohydrodynamics equations with horizontal dissipation and horizontal magnetic diffusion, preprint.

\bibitem{CaoWu} C. Cao and J. Wu, Global regularity for the 2D MHD equations with mixed partial dissipation
and magnetic diffusion, {\it Adv.  Math. \bf 226} (2011), 1803--1822.

\bibitem{CaoWuYuan} C. Cao, J. Wu and B. Yuan, The 2D incompressible
magnetohydrodynamics equations with only magnetic diffusion,
{\it SIAM J. Math. Anal. \bf 46} (2014),  588--602.

\bibitem{CaiLei} Y. Cai and Z. Lei, Global well-posedness of the incompressible magnetohydrodynamics, arXiv: 1605.00439 [math.AP] 2 May 2016.

\bibitem{Chem} J.-Y. Chemin, D.S. McCormick, J.C. Robinson and J.L. Rodrigo, Local existence for the non-resistive MHD equations in Besov spaces, {\it Adv. Math. \bf 286} (2016), 1--31.

\bibitem{ChenWang} G.-Q. Chen and D. Wang, Global solutions of nonlinear magnetohydrodynamics
with large initial data, {\it J. Differential Equations \bf 182} (2002), 344-376.


\bibitem{Con22} P. Constantin, Lagrangian-Eulerian methods for uniqueness in hydrodynamic systems,  {\it Adv. Math. \bf 278} (2015), 67--102.

\bibitem{Davi} P.A. Davidson, {\it An Introduction to Magnetohydrodynamics},  Cambridge University Press, Cambridge, England, 2001.

 \bibitem{DZ}
L. Du and D. Zhou, Global well-posedness of two-dimensional
magnetohydrodynamic flows with partial dissipation and magnetic
diffusion, {\it SIAM J. Math. Anal.  \bf 47}  (2015), 1562-1589.


\bibitem{FNZ} J. Fan, H. Malaikah, S. Monaquel, G. Nakamura and Y. Zhou,
Global cauchy problem of 2D generalized MHD equations, {\it Monatsh.  Math. \bf 175} (2014), 127-131.

\bibitem{Fefferman1}C.L. Fefferman, D.S. McCormick, J.C. Robinson and J.L. Rodrigo, Higher order commutator estimates and local existence for the non-resistive MHD equations and related models, 	 {\it J. Funct. Anal. \bf 267} (2014), 1035--1056.

\bibitem{Fefferman2} C.L. Fefferman, D.S. McCormick, J.C. Robinson and J.L. Rodrigo, Local existence for the non-resistive MHD equations in nearly optimal Sobolev spaces,  {\it Arch. Ration. Mech. Anal. \bf 223} (2017),  677--691.


\bibitem{HeXuYu} L. He, L. Xu and P. Yu, On global dynamics of three dimensional magnetohydrodynamics: nonlinear stability of Alfv\'{e}n waves, 	arXiv:1603.08205 [math.AP] 27 Mar 2016.

\bibitem{Hitt} S. Hittmeir and S. Merino-Aceituno, Kinetic derivation of fractional Stokes and Stokes-Fourier systems, arXiv: 1408.6400v2 [math-ph] 30 Sep 2014.

\bibitem{HuX} X. Hu, Global existence for two dimensional compressible
magnetohydrodynamic flows with zero magnetic diffusivity, arXiv: 1405.0274v1 [math.AP] 1 May 2014.

\bibitem{HuLin} X. Hu and F. Lin, Global Existence for Two Dimensional Incompressible Magnetohydrodynamic Flows with Zero Magnetic Diffusivity, arXiv: 1405.0082v1 [math.AP] 1 May 2014.

\bibitem{HuWang} X. Hu and D. Wang, Global existence and large-time behavior of solutions to the three-dimensional equations of compressible magnetohydrodynamic flows, {\it Arch. Ration. Mech. Anal. \bf 197} (2010), 203-238.

\bibitem{JiuNiu} Q. Jiu and D. Niu, Mathematical results related to a two-dimensional magneto-hydrodynamic equations,
{\it Acta Math. Sci. Ser. B Engl. Ed. \bf 26} (2006), 744-756.

\bibitem{JNW} Q. Jiu, D. Niu, J. Wu, X. Xu and H. Yu, The 2D magnetohydrodynamic equations with magnetic diffusion, {\it Nonlinearity \bf 28} (2015),  3935-3955.

\bibitem{JiuZhao} Q. Jiu and J. Zhao, A remark on global regularity of 2D generalized magnetohydrodynamic equations, {\it J. Math. Anal. Appl. \bf 412} (2014), 478-484.

\bibitem{JiuZhao2} Q. Jiu and J. Zhao, Global regularity of 2D generalized MHD equations with magnetic diffusion, {\it Z. Angew. Math. Phys. \bf 66} (2015), 677-687.

 \bibitem{KP}T.  Kato and G. Ponce,
  Commutator estimates and the Euler and the Navier-Stokes
 equations,  {\it Comm. Pure Appl. Math. \bf 41} (1988),  891--907.

\bibitem{LeiZ} Z. Lei and Y.  Zhou, BKM's criterion and global weak solutions for magnetohydrodynamics with zero viscosity, {\it Discrete Contin. Dyn. Syst. \bf 25} (2009), 575-583.


\bibitem{Lem} P.G. Lemarie-Rieusset,
 {\it Recent Developments in the Navier-Stokes Problem}, Chapman \
Hall/CRC Research Notes in Mathematics Series, CRC Press, 2002.


\bibitem{LZ}
F. Lin, L. Xu and P. Zhang, Global small solutions to an MHD-type
system: the three-dimensional case, {\it Comm. Pure Appl. Math.  \bf 67}
(2014),  531--580.

\bibitem{LinZhang1} F. Lin, L. Xu, and  P. Zhang, Global small solutions to 2-D incompressible MHD system, {\it J. Differential Equations} {\bf 259} (2015), 5440--5485.


\bibitem{Pri} E. Priest and T. Forbes, {\it Magnetic Reconnection, MHD Theory and Applications},
Cambridge University Press, Cambridge, 2000.

\bibitem{Ren} X. Ren, J. Wu, Z. Xiang and Z. Zhang, Global existence and decay of smooth solution for the 2-D MHD equations without magnetic diffusion, {\it J. Functional Analysis \bf 267} (2014),  503--541.





\bibitem{Stein} E. Stein, {\it Singular Integrals and Differentiability Properties of Functions}, Princeton Unviersity Press, Princeton, NJ, 1970.


\bibitem{Stew} K. Stewartson, On asymptotic expansions in the theory of boundary layers, {\it Studies in Applied Mathematics \bf 36} (1957),  173--191.

\bibitem{TrYu} C. Tran, X. Yu and Z. Zhai, On global regularity of 2D generalized magnetohydrodynamic
equations, {\it J. Differential Equations \bf 254} (2013), 4194--4216.


\bibitem{WeiZ} D. Wei and Z. Zhang, Global well-posedness of the MHD equations in a homogeneous magnetic field, 	
    arXiv:1607.04397 [math.AP] 15 Jul 2016.

\bibitem{Wu2} J. Wu,  Generalized MHD equations, {\it J. Differential Equations \bf 195} (2003),  284--312.

\bibitem{Wu3} J. Wu, Regularity criteria for the generalized MHD equations, {\it Comm. Partial Differential Equations \bf 33} (2008), 285--306.

\bibitem{Wu4} J. Wu, Global regularity for a class of generalized magnetohydrodynamic equations, {\it J. Math. Fluid Mech. \bf 13} (2011), 295--305.

\bibitem{WuWu} J. Wu and Y. Wu, Global small solutions to the compressible 2D magnetohydrodynamic system without magnetic diffusion,  {\it Adv. Math. \bf 310} (2017), 759--888.

\bibitem{WuWuXu} J. Wu, Y. Wu and X. Xu, Global small solution to the 2D MHD system with a velocity damping term, {\it SIAM J. Math. Anal. \bf 47} (2015), 2630--2656.

\bibitem{WuZhang} J. Wu and P. Zhang, The global regularity problem
on the 2D magnetohydrodynamic equations with magnetic diffusion only,
work in progress.


\bibitem{Yam1} K. Yamazaki, On the global well-posedness of N-dimensional generalized MHD system in anisotropic spaces, {\it Adv. Differential Equations \bf 19} (2014),  201--224.

\bibitem{Yam2} K. Yamazaki, Remarks on the global regularity of the two-dimensional magnetohydrodynamics system with zero dissipation,
    {\it Nonlinear Anal. \bf 94} (2014), 194--205.

\bibitem{Yam3} K. Yamazaki, On the global regularity of two-dimensional generalized magnetohydrodynamics system, {\it J. Math. Anal. Appl. \bf 416} (2014), 99--111.

\bibitem{Yam4} K. Yamazaki, Global regularity of logarithmically supercritical MHD system with zero diffusivity, {\it Appl. Math. Lett. \bf 29} (2014), 46--51.
    
\bibitem{Yam5} K. Yamazaki, Global regularity of N-dimensional generalized MHD system with anisotropic dissipation and diffusion, {\it Nonlinear Anal. \bf 122} (2015),
176-191.

\bibitem{YuanBai} B. Yuan and L. Bai, Remarks on global regularity of 2D generalized MHD equations, {\it J. Math. Anal. Appl. \bf 413} (2014), 633--640.


\bibitem{ZZ}
C. Zhai and T. Zhang, Global existence and uniqueness theorem to 2-D
incompressible non-resistive MHD system with non-equilibrium
background magnetic field,  {\it J. Differential Equations  \bf 261} (2016),
3519--3550.

\bibitem{TZhang} T. Zhang, An elementary proof of the global existence and uniqueness theorem to 2-D incompressible non-resistive MHD system, arXiv:1404.5681v1 [math.AP] 23 Apr 2014.

\end{thebibliography}
\end{document}